\documentclass[11pt,a4paper]{amsart}
\usepackage[T1]{fontenc}
\usepackage[english]{babel}
\usepackage{lmodern,microtype}
\usepackage[a4paper,margin=28mm]{geometry}
\usepackage{amsmath,amssymb,amsthm,mathtools}
\usepackage{enumitem,booktabs,array}
\usepackage{tikz}
\usepackage{tikz-cd}
\usepackage{float}
\usetikzlibrary{arrows.meta,positioning,calc,matrix,patterns}
\usepackage{xurl}
\usepackage[hidelinks]{hyperref}
\setlist[enumerate]{label=\textnormal{(\roman*)},leftmargin=*,itemsep=3pt,topsep=5pt}
\setlist[itemize]{leftmargin=*,itemsep=3pt,topsep=5pt}
\theoremstyle{plain}
\newtheorem{theorem}{Theorem}[section]
\newtheorem{lemma}[theorem]{Lemma}
\newtheorem{proposition}[theorem]{Proposition}
\newtheorem{corollary}[theorem]{Corollary}
\theoremstyle{definition}

\newtheorem{example}[theorem]{Example}
\newtheorem{question}[theorem]{Question}
\theoremstyle{remark}
\newtheorem{remark}[theorem]{Remark}
\AddToHook{env/theorem/begin}{\begin{samepage}}
\AddToHook{env/theorem/end}{\end{samepage}}
\AddToHook{env/lemma/begin}{\begin{samepage}}
\AddToHook{env/lemma/end}{\end{samepage}}
\AddToHook{env/proposition/begin}{\begin{samepage}}
\AddToHook{env/proposition/end}{\end{samepage}}
\AddToHook{env/corollary/begin}{\begin{samepage}}
\AddToHook{env/corollary/end}{\end{samepage}}
\newcommand{\Vp}{\mathsf V_p}
\newcommand{\Vone}{\mathsf V_1}
\newcommand{\Sone}{\mathsf S_1}
\newcommand{\Sfin}{\mathsf S_{\mathrm{fin}}}
\newcommand{\Ufin}{\mathsf U_{\mathrm{fin}}}
\newcommand{\Gone}{\mathsf G_1}
\newcommand{\Gfin}{\mathsf G_{\mathrm{fin}}}
\newcommand{\Gammap}{\Gamma_p}
\newcommand{\SC}{\mathsf{C}^{\uparrow}}
\newcommand{\CD}{\mathsf{C}}
\newcommand{\Iwin}{\mathsf I\mathord\uparrow}
\newcommand{\Inwin}{\mathsf I\not\uparrow}

\newcommand{\fin}{\mathrm{fin}}
\newcommand{\pos}{\mathrm{pos}}
\newcommand{\cbdl}{\mathrm{lat}}
\newcommand{\covM}{\operatorname{cov}(\mathcal M)}
\newcommand{\ra}{\mathfrak r_a}
\newcommand{\aT}{\mathfrak a_T}

\newcommand{\Bad}{\mathcal Q}
\newcommand{\concat}{\mathbin{{}^\frown}}
\newcommand{\restr}{\mathbin{\upharpoonright}}
\title[Primeless proofs of the Menger and Rothberger games]{Primeless proofs of the Menger and Rothberger games}
\author[R. M. Mezabarba]{Renan M. Mezabarba}
\address{Departamento de Ciências Exatas e Tecnológicas,
Universidade Estadual de Santa Cruz,
Ilhéus, Bahia, Brazil}

\email{rmmezabarba@uesc.br}


\begin{document}
\begin{abstract}
We continue the study of the Menger and Rothberger games on lattices carried out in~\cite{Mezabarba}. This time, we extend the earlier results by dropping some hypotheses that turned out to be unnecessary, and use Stone duality to recover known game characterizations for dense open families. We also give a formulation of $\Ufin$ for partially ordered sets and prove its game characterization without any lattice assumption. Finally, almost disjoint families give complete distributive lattices on which the selection principles and the corresponding games differ. Dias's results give the lower bounds $\covM$ and $\mathfrak d$ for the least sizes of such counterexamples~\cite{Dias}.
\end{abstract}
\maketitle
\section{Introduction}\label{sec:introduction}

Hurewicz's and Pawlikowski's theorems characterize the Menger and Rothberger properties by the absence of winning strategies for Player I in the corresponding games \cite{Hurewicz,Pawlikowski}. In Scheepers' notations introduced in~\cite{ScheepersI}, these are the equivalences between $\Sfin(\mathcal O,\mathcal O)$ and $\Inwin\Gfin(\mathcal O,\mathcal O)$, and between $\Sone(\mathcal O,\mathcal O)$ and $\Inwin\Gone(\mathcal O,\mathcal O)$, where $\mathcal O$ denotes the family of open covers of a topological space. The proofs of Szewczak and Tsaban~\cite{ST} allow one to identify the operations on open covers which are needed in these arguments.

In~\cite{Mezabarba}, the same questions were considered for a lattice $L$. There, enough prime elements were used to replace the arguments in~\cite{ST} involving points.
Although the obtained results generalized the topological case, the restriction concerning prime elements excluded familiar examples, like Boolean algebras and frames.

So, after the publication of that paper, I began studying these questions for Boolean algebras with Fab\'iola Loterio, then my undergraduate research student. Together, we considered arguments using Boolean complements in place of prime elements, but the project remained unfinished. When I returned to these questions in July 2026, discussions with ChatGPT 5.6 Sol helped me recognize how preservation of suprema by binary meets could be used to extend those arguments beyond Boolean algebras. This led me to consider the distributive identities, which are used in the present paper as follows.

We first show that enough prime elements already imply the join-infinite distributive law used in~\cite{Mezabarba}. Next we prove the corresponding lattice versions of Hurewicz's and Pawlikowski's theorems in Sections~\ref{sec:finite} and~\ref{sec:one} under less restrictive distributivity hypotheses.
In Section~\ref{sec:ufin} we return to $\Ufin$ and prove its game characterization, and in Section~\ref{sec:cardinals} we address the limits of the first two game theorems. Finally, in Section~\ref{sec:stone} we use Stone duality to recover known characterizations for dense open families on Stone spaces and state some further questions.

\medskip
\noindent\textit{Acknowledgments.} I would like to thank Fab\'iola Loterio for her work on the Boolean algebra case and for the discussions that motivated the present continuation. I also acknowledge the use of ChatGPT 5.6 Sol for brainstorming, suggestions of cardinal inequalities and counterexamples. Discussions with ChatGPT (GPT-6 Astra) also helped me locate the earlier results of Scheepers and Dias discussed in Remark~\ref{rem:scheepers} and before Theorem~\ref{thm:lower}, which I had overlooked when preparing the first version of this paper. By signing this paper, I take full responsibility for the content and any errors that may remain, as the definition of authorship requires.

\section{Preliminaries and distributive identities}
\label{sec:preliminaries}

Following Scheepers' notations~\cite{ScheepersI}, for a nonempty family $\mathcal B$ of nonempty subsets of a set $D$, the principle $\Sone(\mathcal B,\mathcal B)$ asserts that for every sequence $(A_n)_{n\in\omega}$ in $\mathcal B$ there exists $a_n\in A_n$ for each $n\in\omega$ such that $\{a_n:n\in\omega\}\in\mathcal B$, while the principle $\Sfin(\mathcal B,\mathcal B)$ requires finite sets $F_n\subseteq A_n$ such that $\bigcup_{n\in\omega}F_n\in\mathcal B$. In a play of the game $\Gone(\mathcal B,\mathcal B)$, Player I chooses $A_n\in\mathcal B$ in inning $n$ and Player II chooses $a_n\in A_n$. Player II wins if $\{a_n:n\in\omega\}\in\mathcal B$. In $\Gfin(\mathcal B,\mathcal B)$, Player II chooses a finite subset $F_n\subseteq A_n$ and wins if $\bigcup_{n\in\omega}F_n\in\mathcal B$. A strategy for a player is a function telling her how to play based on the previous moves of her opponent. We write $J\mathord\uparrow G$ when Player $J$ has a winning strategy in $G$, and $J\not\uparrow G$ for its negation, where $J\in\{\mathsf I,\mathsf{II}\}$.

Let $\mathbb P$ be a lattice with greatest element $1$. Following~\cite{Mezabarba}, we denote by $\Vone$ the family of all nonempty subsets of $\mathbb P$ whose supremum is $1$. By analogy with the case in which $\mathbb P$ is a topology, we call the members of $\Vone$ covers of $1$. A nonempty subset belongs to $\Vone$ if and only if no $q<1$ bounds it, and any subset of $\mathbb P$ containing a member of $\Vone$ also belongs to $\Vone$.

In the following, we shall use restricted forms of two distributivity conditions already present in the literature. The \emph{join-infinite distributive law} requires the map $a\mapsto x\wedge a$ to preserve every existing supremum for each $x\in\mathbb P$, see~\cite[Definition~2(ii)]{AC}. For increasing countable sequences, the corresponding identity occurs in the definition of $\aleph_0$-meet-continuity, where the existence of the supremum of every such sequence is also required~\cite[Section~8]{Wehrung}. Here we consider only families with supremum $1$ and, for convenience, use the following abbreviations for these restrictions.
\begin{enumerate}
\item[$\SC_1$:] If $(a_n)_{n\in\omega}$ is increasing and has supremum $1$, then for every $x\in\mathbb P$ the following supremum exists and is equal to $x$:
\[
 \bigvee_{n\in\omega}(x\wedge a_n)=x.
\]

\item[$\CD_1$:] For every $A\in\Vone$ and every $x\in\mathbb P$, the following supremum exists and is equal to $x$:
\[
 \bigvee_{a\in A}(x\wedge a)=x.
\]

\end{enumerate}

We denote by $\CD_{\omega,1}$ the second condition restricted to countable members of $\Vone$, and write $a_n\nearrow1$ when $(a_n)_{n\in\omega}$ is increasing with supremum $1$.

\begin{lemma}\label{lem:conditions}
Let $\mathbb P$ be a lattice with greatest element $1$.
\begin{enumerate}
\item $\CD_1\Rightarrow\CD_{\omega,1}\Rightarrow\SC_1$.
\item If $\mathbb P$ is distributive, then $\SC_1$ implies $\CD_{\omega,1}$.
\item If every cover of $1$ has a countable subcover, then $\CD_{\omega,1}$ implies $\CD_1$.
\end{enumerate}
\end{lemma}
\begin{proof}
Assertion (i) is clear. For (ii), enumerate a countable cover as $(b_n)_{n\in\omega}$ and put $a_n=\bigvee_{j\le n}b_j$. Then $a_n\nearrow1$, so $\SC_1$ applies. For $x\in\mathbb P$, if $u$ bounds all $x\wedge b_j$, finite distributivity yields
\[
 x\wedge a_n=\bigvee_{j\le n}(x\wedge b_j)\le u,
\]
so $\SC_1$ implies $x\le u$. Since $x$ bounds the family of meets, it is its supremum.

For the last assertion, given $A\in\Vone$, take a countable subcover $A_0\subseteq A$. For each $x\in\mathbb P$, the meets of $x$ with the members of $A_0$ have supremum $x$, and adjoining the remaining meets, all below $x$, preserves that supremum.
\end{proof}

Following~\cite{Mezabarba}, an element $r$ of a lattice is \emph{prime} if it is not a greatest element and $a\wedge b\le r$ implies $a\le r$ or $b\le r$. A lattice has \emph{enough prime elements} if, whenever $a\nleq b$, some prime element $r\ge b$ satisfies $a\nleq r$.

\begin{proposition}\label{prop:primes}
Every lattice with greatest element $1$ and enough prime elements satisfies condition $\CD_1$.
\end{proposition}
\begin{proof}
Fix $A\in\Vone$ and $x\in\mathbb P$, and let $u$ be an upper bound for $\{x\wedge a:a\in A\}$. If $x\nleq u$, we might choose a prime element $r\ge u$ with $x\nleq r$. Since $x\wedge a\le r$, primality gives $a\le r$ for every $a\in A$, contradicting $\sup A=1$. Thus $x\le u$, hence $\bigvee_{a\in A}(x\wedge a)=x$.
\end{proof}

In~\cite[Section~2]{Mezabarba}, bounded lattices with enough prime elements were called \emph{pre-Pawlikowski lattices}, and the term \emph{Pawlikowski lattice} was used when distributivity over every existing supremum was additionally assumed. This additional hypothesis is redundant. Indeed, let $v=\bigvee A$ be an existing supremum and fix $x$ in the lattice. The element $x\wedge v$ is an upper bound for $\{x\wedge a:a\in A\}$. If $u$ is another upper bound and $x\wedge v\nleq u$, choosing a prime element $r\ge u$ with $x\wedge v\nleq r$ gives $x\nleq r$ and $v\nleq r$. Since $v=\bigvee A$, there exists $a\in A$ with $a\nleq r$. As $x\wedge a\le u\le r$, primality yields $x\le r$ or $a\le r$, a contradiction. Thus $x\wedge v\le u$, proving that
\[
 x\wedge\bigvee A=\bigvee_{a\in A}(x\wedge a).
\]

\section{Hurewicz's theorem and the Menger game}
\label{sec:finite}

Let $\mathbb P$ be a lattice with greatest element $1$, and suppose that $\Sfin(\Vone,\Vone)$ holds. Fix a strategy $\sigma$ for Player~I in $\Gfin(\Vone,\Vone)$. By the reduction in~\cite[Section~1]{Mezabarba}, we may replace $\sigma$ by a strategy represented by a tree $(c_s)_{s\in\omega^{<\omega}\setminus\{\emptyset\}}$ with the following properties:
\begin{enumerate}
\item\label{item:increasing} For every $s\in\omega^{<\omega}$, the sequence $(c_{s\concat m})_{m\in\omega}$ is increasing with supremum $1$.
\item\label{item:zerochild} For every nonempty sequence $s\in\omega^{<\omega}$, $c_{s\concat0}=c_s$.
\end{enumerate}

In this way, Player~I starts with $\{c_{\langle m\rangle}:m\in\omega\}$, to which Player~II may respond by choosing just one element $c_{\langle m\rangle}$. Player~I then plays $\{c_{\langle m\rangle\concat k}:k\in\omega\}$, and the play continues in the same way.

The following version of~\cite[Lemma~1.2]{Mezabarba}, adapted from~\cite{ST}, replaces the hypothesis on prime elements by $\SC_1$. Here cofiniteness is taken in the set of indices $\omega^n$.

\begin{lemma}\label{lem:tails}
Let $\mathbb P$ be a lattice with greatest element $1$ satisfying $\SC_1$, and let $(c_s)_{s\in\omega^{<\omega}\setminus\{\emptyset\}}$ satisfy conditions~\ref{item:increasing} and~\ref{item:zerochild} above. Then, for every $n\ge1$ and every cofinite subset $J$ of $\omega^n$, the meet
\[
 t_J=\bigwedge_{s\in J}c_s
\]
exists, and the family $T_n=\{t_J:J\text{ is a cofinite subset of }\omega^n\}$
belongs to $\Vone$.
\end{lemma}
\begin{proof}
We follow the induction in~\cite[Lemma~1.2]{Mezabarba}. At level one, condition~\ref{item:increasing} gives $\bigwedge_{m\in J}c_{\langle m\rangle}=c_{\langle\min J\rangle}$ for every cofinite subset $J$ of $\omega$, and hence $T_1\in\Vone$. At the next level, for a cofinite subset $E$ of $\omega^{n+1}$ put $r_s=\min\{m:s\concat m\in E\}$, $D=\{s:r_s>0\}$ and $C=\omega^n\setminus D$. The set $D$ is finite, and conditions~\ref{item:increasing} and~\ref{item:zerochild} give, as in~\cite[Lemma~1.2]{Mezabarba},
\[
 t_E=t_C\wedge\bigwedge_{s\in D}c_{s\concat r_s},
\]
where the last factor is omitted if $D=\emptyset$. The element $t_C$ exists by induction, and only finitely many additional factors occur, so the right side exists because $\mathbb P$ is a lattice. It remains to replace the argument using prime elements to prove that these meets cover $1$.

To show that $T_{n+1}$ covers $1$, suppose that $q<1$ bounds it. By the induction hypothesis, choose a cofinite subset $C_0$ of $\omega^n$ with $t_{C_0}\nleq q$.
Enumerate $\omega^n\setminus C_0$ as $s_1,\ldots,s_k$, and set $d_0=t_{C_0}$. At step $j$, apply $\SC_1$ to $d_{j-1}$ and the sequence $(c_{s_j\concat m})_{m\in\omega}$ to obtain an index $m_j>0$ such that
\[
 d_{j-1}\wedge c_{s_j\concat m_j}\nleq q,
\]
otherwise all these meets, and hence their supremum $d_{j-1}$, would be below $q$. Increasing an index to make it positive preserves the inequality. Put $d_j=d_{j-1}\wedge c_{s_j\concat m_j}$.

Keep all indices $s\concat m$ with $s\in C_0$ and $m\in\omega$, and over each $s_j$ keep precisely those indices $s_j\concat m$ with $m\ge m_j$. The resulting set $E$ is a cofinite subset of $\omega^{n+1}$, and the displayed meet formula gives $t_E=d_k\nleq q$, a contradiction. Thus $T_{n+1}$ has supremum $1$, completing the induction.
\end{proof}

The recursive choice in this proof extends the Boolean algebra argument developed with Fab\'iola Loterio. There, $d_j\nleq q$ was expressed as $d_j\wedge\neg q\ne0$, and successive factors were chosen to keep this meet nonzero. The existence of each choice already followed from the distributive identity $s=\bigvee_{a\in A}(s\wedge a)$ for $s>0$ and $A\in\Vone$. Here the same identity is applied directly to $d_j$ to preserve $d_j\nleq q$, so the recursive argument applies without Boolean complements. Since the families used in these choices are increasing sequences, $\SC_1$ suffices.

\begin{theorem}\label{thm:menger}
Let $\mathbb P$ be a lattice with greatest element $1$ satisfying $\SC_1$. Then $\Sfin(\Vone,\Vone)$ holds if and only if Player~I has no winning strategy in $\Gfin(\Vone,\Vone)$.
\end{theorem}
\begin{proof}
Assume $\Sfin(\Vone,\Vone)$ and fix a strategy for Player~I. Use the reduction above and Lemma~\ref{lem:tails} in the proof of~\cite[Theorem~1.1]{Mezabarba}. The remaining selection argument is unchanged: finite selections from the families $T_n$ determine cofinite sets of indices, through which a branch can be chosen, and its values have supremum $1$. The associated play defeats the strategy. The converse holds in general.
\end{proof}

Examples of lattices for which the equivalence in Theorem~\ref{thm:menger} fails will be given in Section~\ref{sec:realizations}.

\begin{corollary}\label{cor:menger_classes}
Let $\mathbb P$ be a frame, a Heyting algebra or a Boolean algebra, with greatest element $1$. Then $\Sfin(\Vone,\Vone)$ holds if and only if Player~I has no winning strategy in $\Gfin(\Vone,\Vone)$.
\end{corollary}
\begin{proof}
Each of these lattices satisfies $\CD_1$, and hence $\SC_1$. For a frame this follows from its definition. In a Heyting algebra, the map $x\wedge(-)$ is a left adjoint and preserves every existing supremum, and every Boolean algebra is Heyting. Apply Theorem~\ref{thm:menger}.
\end{proof}

\begin{remark}\label{rem:interval}
For an element $p$ of a lattice $L$, let $\mathsf{D}_p$ denote the sublattice of elements below $p$, and let $\Vp$ be the family of nonempty subsets of $L$ with supremum $p$. Replacing $1$ by $p$ and restricting $x$ to $\mathsf{D}_p$ in the preceding conditions defines $\SC_p$, $\CD_p$ and $\CD_{\omega,p}$. The supremum of a subset of $\mathsf{D}_p$ is $p$ in $L$ if and only if it is $p$ in $\mathsf{D}_p$. Consequently, taking $\mathbb P=\mathsf{D}_p$ in Theorem~\ref{thm:menger} gives the equivalence between $\Sfin(\Vp,\Vp)$ and $\Inwin\Gfin(\Vp,\Vp)$ whenever $\SC_p$ holds.
\end{remark}

\begin{remark}\label{rem:scheepers}
After the first version of this work appeared on arXiv, I learned of Scheepers's paper~\cite{ScheepersHurewicz}, which treats more general families in distributive lattices using the language of Hurewicz pairs. For distributive lattices, Theorem~\ref{thm:menger} follows from his Theorem~7, since $\SC_1$ gives the required condition on the linearization of Hurewicz trees. The approach developed here aims to generalize the proof of Szewczak and Tsaban~\cite{ST}, identifying the lattice hypotheses needed for their argument.
\end{remark}

\section{Pawlikowski's theorem and the Rothberger game}
\label{sec:one}

Let us recall the proof strategy used in~\cite[Section~2]{Mezabarba}. First, we use the following form of~\cite[Lemma~2.1]{Mezabarba}.

\begin{lemma}\label{lem:product}
Let $(L_k)_{k\in\omega}$ be lattices with respective greatest elements $1_k$. If each $L_k$ satisfies $\Sfin(\mathsf V_{1_k},\mathsf V_{1_k})$, then $\prod_{k\in\omega}L_k$ satisfies $\Sfin(\Vone,\Vone)$, where $1=(1_k)_{k\in\omega}$. If each $L_k$ satisfies $\SC_{1_k}$, then the product satisfies $\SC_1$.
\end{lemma}

Since the selection argument in~\cite[Lemma~2.1]{Mezabarba} uses no prime elements, it remains valid here. For the additional assertion, an increasing sequence has supremum equal to the product's greatest element if and only if this holds in every coordinate, and $\SC_{1_k}$ applies in coordinate $k$ to its meets with a fixed element.

In the following argument, we may restrict Player~II to nonempty finite responses by replacing empty responses with singletons in the history supplied to the strategy.

\begin{lemma}\label{lem:strong}
Suppose that $\mathbb P$ is a lattice with least element $0$ and greatest element $1$, satisfies $\SC_1$, and satisfies $\Sfin(\Vone,\Vone)$. Every strategy for Player~I in $\Gfin(\Vone,\Vone)$ admits a play with nonempty finite responses $(F_n)_{n\in\omega}$ such that, for every $q<1$, the inequality
\begin{equation}\label{eq:strong}
 \bigvee F_n\nleq q
\end{equation}
holds for infinitely many $n\in\omega$.
\end{lemma}
\begin{proof}
We follow the construction of~\cite[Proposition~2.2]{Mezabarba}. For $a\in\mathbb P$ and $k\in\omega$, let $\delta_k(a)\in\mathbb P^\omega$ have value $a$ at coordinate $k$ and $0$ elsewhere. Replace each move $A$ of the given strategy by
\[
 \widetilde A=\{\delta_k(a):k\in\omega,\ a\in A\}.
\]

Each coordinate projection of this family covers $1$, so $\widetilde A$ covers the greatest element of the product. A finite response in the product determines a finite response to $A$ by taking the nonzero coordinate values, and taking $0$ if the zero element was selected. Supplying these responses to the original strategy defines a strategy in the product.

The product satisfies $\Sfin$ and $\SC_1$ by Lemma~\ref{lem:product}. Thus Theorem~\ref{thm:menger} applies, yielding a defeating play with nonempty finite responses $(E_n)_{n\in\omega}$. Let $(F_n)_{n\in\omega}$ be the corresponding responses in $\mathbb P$. We claim that this sequence has the desired property.

Fix $q<1$ and $N\in\omega$. The elements selected before inning $N$ have only finitely many nonzero coordinates altogether, so choose a coordinate $k$ outside this finite set. Since the selected elements have supremum $1$ in the product, some selected element has its $k$th coordinate not below $q$. This element was selected at an inning $n\ge N$, and its nonzero coordinate value belongs to $F_n$, so $\bigvee F_n\nleq q$. Since $N$ was arbitrary, the required inequality holds infinitely often.
\end{proof}

The following equivalent formulation of $\CD_1$ will be used in place of prime elements in the selection argument. In the next lemma, $\mathord\downarrow r$ denotes the set $\{x\in\mathbb P:x\le r\}$.

\begin{lemma}\label{lem:fiber}
Let $\mathbb P$ be a lattice with greatest element $1$. The condition $\CD_1$ is equivalent to the following property: whenever $d,q\in\mathbb P$ and $d\nleq q$, there is $r<1$ such that
\begin{equation}\label{eq:fiber}
 \{x\in\mathbb P:d\wedge x\le q\}\subseteq\mathord\downarrow r.
\end{equation}

\end{lemma}
\begin{proof}
Assume $\CD_1$ and let $K$ be the family in~\eqref{eq:fiber}. Notice that $q\in K$. If the desired $r<1$ did not exist, then $K$ would cover $1$, and we would have $d=\bigvee_{x\in K}(d\wedge x)\le q$, a contradiction.

Conversely, fix $d\in\mathbb P$, let $A\in\Vone$ and let $y$ be an upper bound of $\{d\wedge a:a\in A\}$. If $d\nleq y$, the hypothesis gives an upper bound $r<1$ for $\{x\in\mathbb P:d\wedge x\le y\}$, which contains $A$, contrary to $\sup A=1$. Hence $d\le y$, proving $\CD_1$.
\end{proof}

\begin{lemma}\label{lem:diagonal}
Let $\mathbb P$ be a lattice with greatest element $1$ satisfying $\CD_1$ and $\Sone(\Vone,\Vone)$. If a sequence $(F_n)_{n\in\omega}$ of nonempty finite subsets of $\mathbb P$ satisfies the conclusion of Lemma~\ref{lem:strong}, then, for each $n\in\omega$, we can choose $f_n\in F_n$ so that $\sup_{n\in\omega}f_n=1$.
\end{lemma}
\begin{proof}
As in~\cite[Lemma~2.3]{Mezabarba}, for $k\in\omega$, let $W_k$ consist of meets of $k+1$ elements drawn from sets $F_n$ with pairwise distinct indices. We claim that $W_k\in\Vone$. To prove this, for each $q<1$ we shall construct a member of $W_k$ not below $q$, using Lemma~\ref{lem:fiber} in place of the prime elements used in~\cite[Lemma~2.3]{Mezabarba}.

Fix $q<1$. Starting with $d_{-1}=1\nleq q$, we recursively use Lemma~\ref{lem:fiber} to find an upper bound $r_j<1$ for the set $\{x\in\mathbb P:d_{j-1}\wedge x\le q\}$. Thus $x\nleq r_j$ implies $d_{j-1}\wedge x\nleq q$. By the hypothesis on $(F_n)_{n\in\omega}$, infinitely many indices $n$ satisfy $\bigvee F_n\nleq r_j$. We may therefore choose $n_j\notin\{n_0,\ldots,n_{j-1}\}$ and $x_j\in F_{n_j}$ with $x_j\nleq r_j$. Then
\[
 d_j=d_{j-1}\wedge x_j\nleq q.
\]

After $k+1$ steps, $d_k\in W_k$ is not below $q$, so no $q<1$ bounds $W_k$ and its supremum is $1$.

The rest of the proof of~\cite[Lemma~2.3]{Mezabarba} applies without change: $\Sone$ allows us to choose $w_k\in W_k$ for each $k\in\omega$ so that $\{w_k:k\in\omega\}\in\Vone$, and then choose one factor from each representation at an index not used in the preceding choices. The $k+1$ distinct indices in the $k$th representation ensure that such an index is available. Complete the remaining choices arbitrarily.
\end{proof}

If $\mathbb P$ has no least element, we may adjoin an element $0_*$ below every member of $\mathbb P$ in order to apply the product construction. Deleting $0_*$ from a cover of $1$ leaves a nonempty cover of $1$, so this extension preserves $\Sone$, $\Sfin$, $\SC_1$ and $\CD_1$. Every strategy whose moves are covers in $\mathbb P$ is also a strategy on the enlarged lattice, and its plays have the same outcomes.

The selection argument above uses $\CD_1$ for covers which need not be increasing, whereas $\SC_1$ sufficed for Hurewicz's theorem. In the next theorem, it is enough to assume $\CD_{\omega,1}$: under $\Sone$, every cover has a countable subcover, so Lemma~\ref{lem:conditions} gives $\CD_1$. Example~\ref{ex:binary_pawlikowski}, after Corollary~\ref{cor:distributive}, will show that $\SC_1$ alone could not replace this hypothesis.

\begin{theorem}\label{thm:rothberger}
Let $\mathbb P$ be a lattice with greatest element $1$. If $\CD_{\omega,1}$ holds, then $\Sone(\Vone,\Vone)$ holds if and only if Player~I has no winning strategy in $\Gone(\Vone,\Vone)$.
\end{theorem}
\begin{proof}
Assume $\Sone(\Vone,\Vone)$ and fix a strategy for Player~I in $\Gone(\Vone,\Vone)$. Every cover has a countable subcover, so Lemma~\ref{lem:conditions} gives $\CD_1$. We may adjoin a least element if necessary, as explained above.

We use the construction of Szewczak and Tsaban~\cite{ST} invoked in~\cite[Theorem~2.4]{Mezabarba}, with Lemmas~\ref{lem:strong} and~\ref{lem:diagonal} in place of~\cite[Proposition~2.2 and Lemma~2.3]{Mezabarba}. The additional distributivity used there guarantees that finite common refinements of covers are covers. Here $\CD_1$ suffices: for $A,B\in\Vone$, if $u$ bounds every $a\wedge b$ with $a\in A$ and $b\in B$, then $a=\bigvee_{b\in B}(a\wedge b)\le u$ for every $a\in A$, so $u=1$. Iterating gives the required finite refinements.

Thus the construction gives an auxiliary strategy in $\Gfin$, to which Lemma~\ref{lem:strong} applies. Lemma~\ref{lem:diagonal} selects one element from each finite response with supremum $1$, and the corresponding factors determine a legal play against the original strategy, with each chosen factor above that selected element. This play also has supremum $1$. The converse holds in general.
\end{proof}

\begin{corollary}\label{cor:distributive}
Let $\mathbb P$ be a distributive lattice with greatest element $1$ satisfying $\SC_1$. Then $\Sone(\Vone,\Vone)$ holds if and only if Player~I has no winning strategy in $\Gone(\Vone,\Vone)$. In particular, this equivalence holds for frames, Heyting algebras and Boolean algebras.
\end{corollary}
\begin{proof}
Lemma~\ref{lem:conditions} gives $\CD_{\omega,1}$, so Theorem~\ref{thm:rothberger} applies. For the indicated classes, $\CD_1$ holds by the argument of Corollary~\ref{cor:menger_classes}.
\end{proof}

As in Remark~\ref{rem:interval}, applying Theorem~\ref{thm:rothberger} to $\mathsf{D}_p$ gives the equivalence for $\Gone(\Vp,\Vp)$ at any element $p$ of a lattice satisfying $\CD_{\omega,p}$. The same observation applies to the corollary.

The next example shows that the hypothesis of Theorem~\ref{thm:rothberger} cannot be replaced by $\SC_1$. The strategy of offering extensions of the preceding response is similar to the branch construction in~\cite{ABD}. Here the infinite branches are themselves elements of the lattice.

\begin{example}\label{ex:binary_pawlikowski}
Let $L=2^{<\omega}\cup2^\omega\cup\{1\}$, where $1$ is an additional greatest element and the binary sequences are ordered by extension. Then $L$ is a lattice of cardinality $\mathfrak c$ satisfying $\SC_1$ and $\Sone(\Vone,\Vone)$, but Player~I has a winning strategy in $\Gone(\Vone,\Vone)$.
\end{example}
\begin{proof}
The least element is the empty sequence. Comparable sequences have their larger element as join and their smaller element as meet, whereas incomparable sequences have join $1$ and meet equal to their longest common initial segment. A nonempty family $A$ covers $1$ if and only if it contains $1$ or two incomparable sequences. Indeed, otherwise $A$ is a chain of sequences and its supremum is $\bigcup A\in2^{<\omega}\cup2^\omega$. An increasing sequence with supremum $1$ must therefore have a term equal to $1$, so $\SC_1$ holds trivially.

For two covers $A,B\in\Vone$, we claim that there are $a\in A$ and $b\in B$ with $a\vee b=1$. If either cover contains $1$, this is immediate. Otherwise choose incomparable $a_0,a_1\in A$. If no such pair $a,b$ existed, every $b\in B$ would be comparable with both $a_0$ and $a_1$, hence $b\le a_0\wedge a_1<1$, contradicting $\sup B=1$. Given a sequence $(A_n)_{n\in\omega}$ of covers, choose such elements from $A_0$ and $A_1$, and make arbitrary choices in the remaining covers. Their supremum is $1$, proving $\Sone$.

Player~I starts by offering $A_0=\{\langle0\rangle,\langle1\rangle\}$. After Player~II chooses $s_0\in A_0$, the next move is $A_1=\{s_0\concat0,s_0\concat1\}$. In general, after a response $s_n$, Player~I offers
\[
 A_{n+1}=\{s_n\concat0,s_n\concat1\}.
\]

The two sequences in each move are incomparable, so it is a legal cover. This strategy forces all responses to lie on a single branch $f=\bigcup_{n\in\omega}s_n\in2^\omega$, and their supremum is $f<1$. Thus Player~I wins.

The failure of $\CD_{\omega,1}$ can also be seen directly. With $x=\langle0\rangle$ and $A=\{\langle1,0\rangle,\langle1,1\rangle\}$, we have $\bigvee A=1$, whereas $x\wedge a=\emptyset$ for every $a\in A$.
\end{proof}

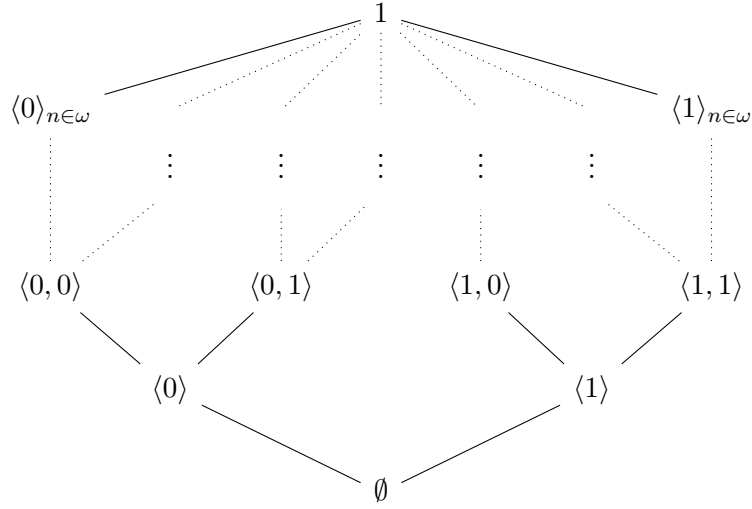
\begin{figure}[htbp]
\centering
\[\begin{tikzcd}[column sep=small]
	&&& 1 &&& \\
	{\langle0\rangle_{n\in\omega}} & {} & {} & {} & {} & {} & {\langle1\rangle_{n\in\omega}} \\
	& {} & {} & {} & {} & {} \\
	{\langle0,0\rangle} && {\langle0,1\rangle} && {\langle1,0\rangle} && {\langle1,1\rangle} \\
	& {\langle0\rangle} &&&& {\langle1\rangle} \\
	&&& \emptyset
	\arrow[dotted, no head, from=1-4, to=2-2]
	\arrow[dotted, no head, from=1-4, to=2-4]
	\arrow[dotted, no head, from=1-4, to=2-5]
	\arrow[dotted, no head, from=1-4, to=2-6]
	\arrow[no head, from=2-1, to=1-4]
	\arrow["\vdots"{description}, draw=none, from=2-2, to=3-2]
	\arrow[dotted, no head, from=2-3, to=1-4]
	\arrow["\vdots"{description}, draw=none, from=2-3, to=3-3]
	\arrow["\vdots"{description}, draw=none, from=2-5, to=3-5]
	\arrow[no head, from=2-7, to=1-4]
	\arrow["\vdots"{description}, draw=none, from=3-4, to=2-4]
	\arrow["\vdots"{description}, draw=none, from=3-6, to=2-6]
	\arrow[dotted, no head, from=4-1, to=2-1]
	\arrow[dotted, no head, from=4-1, to=3-2]
	\arrow[dotted, no head, from=4-3, to=3-3]
	\arrow[dotted, no head, from=4-3, to=3-4]
	\arrow[dotted, no head, from=4-5, to=3-5]
	\arrow[dotted, no head, from=4-7, to=2-7]
	\arrow[dotted, no head, from=4-7, to=3-6]
	\arrow[no head, from=5-2, to=4-1]
	\arrow[no head, from=5-2, to=4-3]
	\arrow[no head, from=5-6, to=4-5]
	\arrow[no head, from=5-6, to=4-7]
	\arrow[no head, from=6-4, to=5-2]
	\arrow[no head, from=6-4, to=5-6]
\end{tikzcd}\]
\caption{The order in Example~\ref{ex:binary_pawlikowski}.}
\label{fig:binary_lattice}
\end{figure}

\section{The Hurewicz property}\label{sec:ufin}

The final paragraph of~\cite{Mezabarba} proposes a version of $\Ufin(\mathcal O,\Gamma)$ for lattices with enough prime elements. For a sequence of finite selections $(F_n)_{n\in\omega}$, the requirement is that their union have supremum $p$ and that $\bigvee F_n\nleq r$ for all but finitely many $n$, for every prime element $r$ with $p\nleq r$. We shall give an equivalent formulation under the hypothesis of enough prime elements and then prove a game characterization for arbitrary posets.

Let $\mathbb P$ be a poset and $p\in\mathbb P$. Write $\mathsf{D}_p=\{x\in\mathbb P:x\le p\}$ and let $\mathsf V_p(\mathbb P)$ be the family of nonempty subsets of $\mathbb P$ with supremum $p$ computed in $\mathbb P$. Every member of $\mathsf V_p(\mathbb P)$ is therefore contained in $\mathsf D_p$. We denote by $\Gammap$ the collection of sequences $(F_n)_{n\in\omega}$ of finite subsets of $\mathsf{D}_p$ such that, for every infinite subset $J$ of $\omega$,
\begin{equation}\label{eq:gamma}
 \bigcup_{n\in J}F_n\in\Vp.
\end{equation}

We write $\Ufin(\Vp,\Gammap)$ if every sequence $(A_n)_{n\in\omega}$ in $\Vp$ admits finite selections $F_n\subseteq A_n$ such that $(F_n)_{n\in\omega}\in\Gammap$.

\begin{proposition}\label{prop:gamma_test}
A sequence $(F_n)_{n\in\omega}$ of finite subsets of $\mathsf{D}_p$ belongs to $\Gammap$ if and only if both of the following conditions hold.
\begin{enumerate}
\item The sets $F_n$ are nonempty for all but finitely many $n$.
\item For every $q\in\mathbb P$ with $p\nleq q$, there is an element of $F_n$ not below $q$ for all but finitely many $n$.
\end{enumerate}
If $\mathbb P$ is a lattice with enough prime elements, it suffices in \textup{(ii)} to consider the prime elements $r$ for which $p\nleq r$.
\end{proposition}
\begin{proof}
If $J=\{n\in\omega:F_n=\emptyset\}$ were infinite, then $\bigcup_{n\in J}F_n=\emptyset\notin\Vp$, contrary to~\eqref{eq:gamma}. Infinitely many selections contained in $\mathord\downarrow q$ for some $q$ with $p\nleq q$ would also violate~\eqref{eq:gamma}. Conversely, the two conditions ensure that the union in~\eqref{eq:gamma} is nonempty and every upper bound of it lies above $p$, so its supremum is $p$. For the last assertion, separate $p\nleq q$ by a prime element $r\ge q$ with $p\nleq r$. An element not below $r$ is not below $q$.
\end{proof}

\begin{remark}
For nonempty finite selections in a lattice, condition (ii) says that $\bigvee F_n\nleq q$ for all but finitely many $n$, for every $q$ with $p\nleq q$. If the lattice has enough prime elements, each such $q$ lies below a prime element $r$ with $p\nleq r$, so it suffices to require these inequalities for prime elements. This recovers the formulation in~\cite{Mezabarba} for nonempty selections. Without enough prime elements, inequalities involving only primes need not detect all upper bounds which prevent a family from covering $p$, which is why the present definition considers all of them.
\end{remark}

Recall that an open cover $\mathcal U$ of a space $X$ is a $\gamma$-cover if it is infinite and every point of $X$ belongs to all but finitely many members of $\mathcal U$, see~\cite[Section~1]{ScheepersI}. The family of these covers is denoted by $\Gamma$. In contrast, $\Gamma_p$ consists of sequences of finite selections, so repetitions at different indices matter.

If $\mathbb P=\mathcal O(X)$ and $V$ is a nonempty open set of $X$, condition~\eqref{eq:gamma}, with $p=V$, says that every $x\in V$ belongs to $\bigcup F_n$ for all but finitely many $n$. Indeed, infinitely many omissions of $x$ give an infinite set of innings whose selected union omits $x$. Conversely, eventual membership ensures that every infinite set of innings covers $V$. Thus~\eqref{eq:gamma} gives the usual indexed formulation of the Hurewicz property \cite{ScheepersI}.

The usual formulation of $\Ufin(\mathcal O,\Gamma)$ considers open covers with no finite subcover, see~\cite{ScheepersI}. The corresponding restricted principle here asserts that, for every sequence $(A_n)_{n\in\omega}$ in $\Vp$ such that no $A_n$ has a finite subfamily belonging to $\Vp$, there are finite subsets $F_n\subseteq A_n$ for which $\bigcup_{n\in J}F_n\in\Vp$ for every infinite subset $J$ of $\omega$. We claim that this is equivalent to $\Ufin(\Vp,\Gammap)$ as defined above.

One implication follows by restricting the sequences of covers. For the converse, let $(A_n)_{n\in\omega}$ be an arbitrary sequence in $\Vp$ and let $K$ be the set of innings in which $A_n$ has a finite subfamily $H_n\in\Vp$. Make that selection at each $n\in K$. If $\omega\setminus K$ is infinite, enumerate it and apply the restricted principle there. If it is finite, choose arbitrary finite selections at its indices.

Fix an infinite subset $J$ of $\omega$. If $J$ meets $K$, the union over $J$ contains a finite cover of $p$ and therefore belongs to $\Vp$. If $J\cap K=\emptyset$, then $\omega\setminus K$ is infinite and the restricted principle gives the same conclusion. Hence the full sequence belongs to $\Gammap$.

For open covers without finite subcovers, the unions of selections satisfying the indexed condition form a classical $\gamma$-cover. No proper union can occur infinitely often, since a point outside it would then be omitted infinitely often. Conversely, the usual formulation of $\Ufin(\mathcal O,\Gamma)$ yields indexed selections after passing to finite common refinements.

We denote by $\Gfin(\Vp,\Gammap)$ the game with the same moves as $\Gfin(\Vp,\Vp)$ in which Player~II wins if the sequence of finite responses belongs to $\Gammap$.

\begin{theorem}\label{thm:ufin}
For every poset $\mathbb P$ and every $p\in\mathbb P$, the principle $\Ufin(\Vp,\Gammap)$ holds if and only if Player~I has no winning strategy in $\Gfin(\Vp,\Gammap)$.
\end{theorem}
\begin{proof}
A sequence witnessing failure of the principle gives a winning strategy for Player~I by playing its successive covers. Conversely, assume $\Ufin(\Vp,\Gammap)$ and fix a strategy $\sigma$ for Player~I. Applying the principle to a constant sequence shows that every cover of $p$ has a countable subcover.

At each finite history $s$, choose a countable subset $C_s\subseteq\sigma(s)$ with $C_s\in\Vp$ and continue the construction after each finite response $D\subseteq C_s$. The resulting tree of histories is countable. Enumerate its nodes without repetitions as $(s_i)_{i\in\omega}$ and apply $\Ufin$ to $(C_{s_i})_{i\in\omega}$, obtaining finite subsets $D_i\subseteq C_{s_i}$ with $(D_i)_{i\in\omega}\in\Gammap$.

Follow a branch by answering with $D_i$ whenever its current history is $s_i$. This response belongs to the retained tree, so the recursion continues. Let $(i_n)_{n\in\omega}$ be the indices of the visited histories. They are pairwise distinct, since their lengths strictly increase. For every infinite subset $J$ of $\omega$, the set $\{i_n:n\in J\}$ is infinite. Consequently,
\[
 \bigcup_{n\in J}D_{i_n}\in\Vp.
\]

Thus the sequence of actual responses belongs to $\Gammap$ and defeats $\sigma$.
\end{proof}

For sets of reals, Scheepers proved that the Hurewicz property is equivalent to Player~I having no winning strategy in the Hurewicz game~\cite[Theorem~27]{ScheepersI}. The proof of Theorem~\ref{thm:ufin} follows the same argument: we make finite selections at all nodes of a countable tree of finite histories and then follow a branch, using the indexed winning condition~\eqref{eq:gamma}. The same equivalence holds when Player~I is restricted to covers without finite subcovers, since the corresponding selection principle is equivalent to the full one, as shown above, and the proof defeats every strategy on the restricted domain. If there are no such covers, the principle is vacuous and Player~I has no legal first move whatsoever.

The implication from the Hurewicz property to the Menger property also holds for this formulation: $\Ufin(\Vp,\Gammap)$ implies $\Sfin(\Vp,\Vp)$ for every poset $\mathbb P$ and $p\in\mathbb P$. In fact, Theorem~\ref{thm:ufin} gives the stronger conclusion $\Inwin\Gfin(\Vp,\Vp)$, since every sequence in $\Gammap$ has its full union in $\Vp$.
The next section gives complete distributive lattices in which $\Sfin$ holds but Player~I wins $\Gfin$, so $\Ufin$ fails by Theorem~\ref{thm:ufin}.

\section{The least cardinalities of counterexamples}
\label{sec:realizations}\label{sec:cardinals}

We now consider how Theorems~\ref{thm:menger} and~\ref{thm:rothberger} can fail without their hypotheses. The examples below show that for each $\#\in\{1,\fin\}$ there is a complete distributive lattice satisfying $\mathsf S_\#(\Vone,\Vone)$ on which Player~I wins $\mathsf G_\#(\Vone,\Vone)$. We shall then estimate the least cardinalities of such counterexamples.

Following the terminology and notation of~\cite[Section~I]{Hrusak}, throughout this section all ideals $\mathcal I$ on a countably infinite set $D$ are proper and contain $[D]^{<\omega}$. Let us set $\mathcal I^+=\mathcal P(D)\setminus\mathcal I$, the collection of subsets of $D$ outside $\mathcal I$, which we call $\mathcal I$-positive sets. Recall that a family $\mathcal A\subseteq[D]^\omega$ is almost disjoint if distinct members have finite intersection. Finally, for an infinite almost disjoint family $\mathcal A$, the collection
\[
 \mathcal I(\mathcal A)=\{S\subseteq D:S\subseteq^*\bigcup\mathcal F
                         \text{ for some }\mathcal F\in[\mathcal A]^{<\omega}\}
\]
is a proper ideal on $D$ containing all finite subsets of $D$. 
Clearly, an infinite subset $S$ of $D$ whose intersection with each member of $\mathcal A$ is finite is $\mathcal I(\mathcal A)$-positive. This will be useful below.

The following selection property is a consequence of the selectivity of $\mathcal I(\mathcal A)^+$, due to Mathias~\cite{Mathias}, see also~\cite[Section~I]{Hrusak}.

\begin{lemma}\label{lem:adselection}
The selection principle
$\Sone(\mathcal I(\mathcal A)^+,\mathcal I(\mathcal A)^+)$ holds for every infinite almost disjoint family $\mathcal A$ on a countably infinite set $D$.
\end{lemma}
\begin{proof}
Let $(X_n)_{n\in\omega}$ be a sequence in $\mathcal I(\mathcal A)^+$ and consider a partition $(J_k)_{k\in\omega}$ of $\omega$ into infinite subsets, and enumerate each $J_k$ as $(j(k,m))_{m\in\omega}$. We have to choose $x_n \in X_n$ for each $n\in\omega$ in such a way that $Y=\{x_n:n\in\omega\}$ is positive. In order to accomplish this, we shall write $Y=\bigcup_{k\in\omega}Y_k$, with $Y_k=\{x_n:n\in J_k\}$ for each $k\in\omega$, and the construction of these latter sets will be done by induction on $k$.

For $k=0$, choose the points $x_{j(0,m)}\in X_{j(0,m)}$ recursively on $m$, avoiding the finitely many points already chosen in previous steps. Since each $X_n$ is positive and therefore infinite, these choices are possible and $Y_0=\{x_n:n\in J_0\}$ is infinite. If $Y_0$ is positive, complete all remaining choices of $Y_k$ with $k>0$ arbitrarily, as this ensures that $Y$ is positive. Otherwise, we may choose a finite subfamily $\mathcal F_0\subseteq\mathcal A$ such that $Y_0\subseteq^*\bigcup\mathcal F_0$.

Suppose now that the choices have been made on $J_0,\ldots,J_k$ and that the finite subfamilies $\mathcal F_0,\ldots,\mathcal F_k\subseteq\mathcal A$ are pairwise disjoint. As each $\mathcal F_i$ is a finite subset of $\mathcal A$, it follows that $\mathcal F=\bigcup_{i\leq k}\mathcal F_i$ is finite as well, hence $X_n\setminus \left(H\cup \bigcup\mathcal F\right)$ is infinite for all $n$ and for every finite subset $H$ of $D$. This allows us to recursively choose
\[
x_{j(k+1,m)}\in X_{j(k+1,m)}\setminus
\bigl(\bigcup\mathcal F\cup\{x_{j(k+1,r)}:r<m\}\bigr).
\]
Now, the set $Y_{k+1}=\{x_n:n\in J_{k+1}\}$ is infinite and disjoint from $\bigcup\mathcal F$. If it is positive, complete the remaining choices arbitrarily. Otherwise, choose a finite subfamily $\mathcal F_{k+1}\subseteq\mathcal A$ witnessing $Y_{k+1}\in \mathcal I(\mathcal A)$. Since removing the members of $\mathcal F$ from $\mathcal F_{k+1}$ does not affect the witnessing property, the finite subfamilies remain pairwise disjoint.

In the worst case, i.e., when no $Y_k$ is positive, we have to check whether the set $Y=\{x_n:n\in\omega\}=\bigcup_{k\in\omega}Y_k$ is positive. Assuming the contrary, let us choose a finite subfamily $\mathcal H\subseteq\mathcal A$ such that $Y\subseteq^*\bigcup\mathcal H$. Since the families $\mathcal F_k$ are pairwise disjoint, some $\mathcal F_k$ is disjoint from $\mathcal H$. Almost disjointness implies that
\[
\left(\bigcup\mathcal F_k\right)\cap\left(\bigcup\mathcal H\right)
\]
is finite. But the infinite set $Y_k$ is almost contained in both unions, a contradiction, as no infinite set can be almost contained in a finite set. Hence $Y$ is positive, as required.
\end{proof}

The lattice that we shall use to transfer selection principles and games from ideals to lattices is essentially the closed-set lattice appearing in an example of Gotchev and Minchev~\cite{GotchevMinchev}. Let $D$ be countably infinite and let $\mathcal A\subseteq[D]^\omega$ be an infinite almost disjoint family. Put
\[
\mathcal C(\mathcal A)=
\{F\cup\bigcup\mathcal F:
F\in[D]^{<\omega},\ 
\mathcal F\in[\mathcal A]^{<\omega}\}
\]
and $L_{\mathcal A}=\mathcal C(\mathcal A)\cup\{D\}$. Gotchev and Minchev consider the corresponding construction for a maximal almost disjoint family, with an additional point playing the role of the greatest element. Maximality, however, is not needed for the lattice construction considered here\footnote{This point is also discussed at
\url{https://math.stackexchange.com/questions/4810010}.}.
Under inclusion, $L_{\mathcal A}$ is a complete bounded distributive lattice of cardinality $|\mathcal A|$, and its finite joins and meets are unions and intersections.

Moreover, the members of $\mathcal I(\mathcal A)$ are precisely the subsets of members of $\mathcal C(\mathcal A)$. Consequently, for every nonempty family $\mathcal U\subseteq L_{\mathcal A}$,
\begin{equation}\label{eq:positivecover}
\sup_{L_{\mathcal A}}\mathcal U=D
\quad\Longleftrightarrow\quad
\bigcup\mathcal U\in\mathcal I(\mathcal A)^+.
\end{equation}

Thus covers of the greatest element of $L_{\mathcal A}$ correspond naturally to positive sets of the ideal generated by $\mathcal A$. The next theorem makes this correspondence precise for both selection principles and games.
Before stating it, recall that two games are said to be equivalent if, for each player, there is a winning strategy in one game if and only if there is a winning strategy in the other, see~\cite{AD}. We write $\mathsf V_D=\mathsf V_D(L_{\mathcal A})$ below.

\begin{theorem}\label{thm:realization}
For each $\#\in\{1,\fin\}$, the following hold.
\begin{enumerate}
\item $\mathsf S_\#(\mathsf V_D,\mathsf V_D)$ holds if and only if $\mathsf S_\#(\mathcal I(\mathcal A)^+,\mathcal I(\mathcal A)^+)$ holds.
\item The games $\mathsf G_\#(\mathsf V_D,\mathsf V_D)$ and $\mathsf G_\#(\mathcal I(\mathcal A)^+,\mathcal I(\mathcal A)^+)$ are equivalent.
\end{enumerate}
\end{theorem}
\begin{proof}
By~\eqref{eq:positivecover}, the correspondences $X\mapsto\widehat{X}=\{\{x\}:x\in X\}$ and $\mathcal U\mapsto\bigcup\mathcal U$ yield functions $\mathcal{I}(\mathcal{A})^+\to\mathsf V_D$ and $\mathsf V_D\to\mathcal{I}(\mathcal{A})^+$, respectively, which will be used to transfer selections and strategies between the two settings.
In the following, for every $\mathcal U\in\mathsf V_D$ and $x\in\bigcup\mathcal U$, fix $U(\mathcal U,x)\in\mathcal U$ containing $x$.

Notice that members of $\widehat X$ correspond exactly to elements of $X$, while finite subfamilies of $\widehat X$ correspond exactly to finite subsets of $X$ by taking unions. Thus, selections from $\widehat X$ may be identified with selections from $X$, and the union of the selected singletons is precisely the corresponding set of selected points. Conversely, if $x\in\bigcup\mathcal U$, then $U(\mathcal U,x)$ is a member of $\mathcal U$ containing $x$. Hence a one-point or finite selection from $\bigcup\mathcal U$ can be replaced by a corresponding one-point or finite selection from $\mathcal U$, whose union contains the selected points.
%
%
Assertions (i) and (ii) follow easily from these observations.
%
%
\end{proof}

The first construction below gives a lattice satisfying $\Sone$ together with a winning strategy for Player~I in $\Gone$, while the second does the same for $\Sfin$ and $\Gfin$. The first is the branch construction used in~\cite{ABD}, and the second is a version of Scheepers's example~\cite{ScheepersIII}, also presented in~\cite{ABD}.

\begin{proposition}\label{prop:branches}
On $D=\omega^{<\omega}$, let $\mathcal A=\{B_f:f\in\omega^\omega\}$, where $B_f=\{f\restr n:n\in\omega\}$. Then $\Sone(\mathcal I(\mathcal A)^+,\mathcal I(\mathcal A)^+)$ holds and Player~I wins $\Gone(\mathcal I(\mathcal A)^+,\mathcal I(\mathcal A)^+)$. Consequently $L_{\mathcal A}$ satisfies $\Sone(\Vone,\Vone)$ at its greatest element but Player~I wins $\Gone(\Vone,\Vone)$.
\end{proposition}
\begin{proof}
Distinct branches have finite intersection, so Lemma~\ref{lem:adselection} applies. For the game $\Gone$, Player~I starts by offering $\{\langle m\rangle:m\in\omega\}$. After Player~II chooses $\langle m_0\rangle$, Player~I offers $\{\langle m_0,m\rangle:m\in\omega\}$. In general, after finitely many moves by Player~II, the last chosen node $s$ records all these choices, and Player~I plays $\{s\concat m:m\in\omega\}$. Each move is an infinite antichain and meets each branch in at most one point. It is therefore positive by the observation above, whereas the selected nodes lie on a single branch.

The lattice counterparts follow from Theorem~\ref{thm:realization}.
\end{proof}

\begin{proposition}\label{prop:scheepers}
There is an almost disjoint family $\mathcal A$ of size $\mathfrak c$ on a countable set such that $\Sfin(\mathcal I(\mathcal A)^+,\mathcal I(\mathcal A)^+)$ holds and Player~I wins $\Gfin(\mathcal I(\mathcal A)^+,\mathcal I(\mathcal A)^+)$. Consequently there is a complete bounded distributive lattice of size $\mathfrak c$ satisfying $\Sfin$ at its top on which Player~I wins $\Gfin$.
\end{proposition}
\begin{proof}
Let $R=[\omega]^{<\omega}\setminus\{\emptyset\}$ and $D=R^{<\omega}\times\omega$. For $f\in R^\omega$, put
\[
 S_f=\{(f\restr n,j):n\in\omega,\ j\in f(n)\},
 \qquad \mathcal A=\{S_f:f\in R^\omega\}.
\]

Each $S_f$ is infinite, and if $f\ne g$, only finitely many histories are shared. Each of them contributes finitely many points, so $S_f\cap S_g$ is finite and the family has size $\mathfrak c$. For every $s\in R^{<\omega}$, the set $D_s=\{s\}\times\omega$ has finite intersection with each generator and is therefore positive.

Player~I starts at the empty history $s$ and offers $D_s$. A nonempty finite response determines $r\in R$ and changes the history to $s\concat r$. An empty response leaves the history unchanged. If there are only finitely many nonempty responses, their union is finite. Otherwise those responses determine $f\in R^\omega$, and their union is $S_f$. In both cases the outcome belongs to the ideal, so Player~I wins. Apply Lemma~\ref{lem:adselection} and Theorem~\ref{thm:realization}.
\end{proof}

These examples give counterexamples of size $\mathfrak c$. We next ask how small a poset, or a complete distributive lattice, can be if a selection principle holds but Player~I wins the corresponding game.

For $\#\in\{1,\fin\}$, define
\[
 \epsilon_\#^{\pos}=\min\{|\mathbb P|:\mathsf S_\#(\Vone,\Vone)\text{ holds and }\Iwin\mathsf G_\#(\Vone,\Vone)\},
\]
where $\mathbb P$ ranges over posets with greatest element $1$. Define $\epsilon_\#^{\cbdl}$ by restricting $\mathbb P$ to complete bounded distributive lattices. Propositions~\ref{prop:branches} and~\ref{prop:scheepers} show that all four minima exist and are at most $\mathfrak c$.

In the first arXiv version of this paper, we used the cardinal characteristics $\covM$, the covering number of the meager ideal, and $\mathfrak d$, the dominating number, to prove $\covM\le\epsilon_1^{\pos}$ and $\mathfrak d\le\epsilon_{\fin}^{\pos}$. Only after that version appeared did I learn that the more general result below follows from theorems of Dias~\cite[Theorems~3.2.4 and~3.2.5]{Dias}.

Dias considers families of subsets of a set $T$ recognized by a relation $R\subseteq H\times T$: a subset belongs to the family if, for each $q\in H$, it contains an element $a$ with $qRa$. In our case, take $T=\mathsf D_p$ and $H=\Bad(\mathbb P,p)=\{q\in\mathbb P:p\nleq q\}$, with $qRa$ meaning $a\nleq q$. If $H$ is nonempty, a subset of $T$ covers $p$ exactly when each $q\in H$ is related to one of its members. Since $|H|\le|\mathbb P|$, Dias's theorems give the assertions below. When $H$ is empty, $p$ is least and the assertions are immediate. As I have not found an English version of these results, we include a direct proof in the language of posets.

\begin{theorem}[Dias~\cite{Dias}]\label{thm:lower}
For every poset $\mathbb P$ and $p\in\mathbb P$:
\begin{enumerate}
\item If $|\mathbb P|<\covM$ and $\Sone(\Vp,\Vp)$ holds, then $\Inwin\Gone(\Vp,\Vp)$.
\item If $|\mathbb P|<\mathfrak d$ and $\Sfin(\Vp,\Vp)$ holds, then $\Inwin\Gfin(\Vp,\Vp)$.
\end{enumerate}
In particular, $\covM\le\epsilon_1^{\pos}$ and $\mathfrak d\le\epsilon_{\fin}^{\pos}$.
\end{theorem}
\begin{proof}
Under either selection hypothesis, every cover of $p$ has a countable subcover. Fix a strategy $\sigma$ for Player~I and restrict the responses to such subcovers. For the game $\Gone$, we adapt the category argument of Hru\v{s}\'ak~\cite[Proposition~II.1]{Hrusak}. Enumerate a countable subcover of $\sigma(\emptyset)$ as $(a_{\emptyset,j})_{j\in\omega}$. If Player~II chooses $a_{\emptyset,j_0}$, enumerate a countable subcover of the next move prescribed by $\sigma$ as $(a_{\langle j_0\rangle,j})_{j\in\omega}$. Continuing recursively, each $s\in\omega^{<\omega}$ records the indices of the preceding choices, and $(a_{s,j})_{j\in\omega}$ enumerates a countable subcover of the corresponding move. Choosing $a_{s,j}$ extends the history of indices to $s\concat j$. For $q\in\Bad(\mathbb P,p)$, put
\[
 C_q=\{f\in\omega^\omega:a_{(f\restr n),f(n)}\le q\text{ for every }n\in\omega\}.
\]

The set $C_q$ is closed. Indeed, if $f\notin C_q$, then for some $n$ we have $a_{(f\restr n),f(n)}\nleq q$, and every branch extending $f\restr(n+1)$ has the same property. It is nowhere dense because, for every finite history $s$, the next move has supremum $p\nleq q$, so some $a_{s,j}\nleq q$; the basic neighbourhood determined by $s\concat\langle j\rangle$ is then disjoint from $C_q$. If the strategy won, the selections along every branch would be bounded by some $q\in\Bad(\mathbb P,p)$, so the sets $C_q$ would cover the Baire space. Since $\covM$ is the least size of a cover of $\omega^\omega$ by closed nowhere dense sets~\cite[Section~II]{Hrusak}, this would require $|\mathbb P|\ge\covM$.

For the game $\Gfin$, enumerate countable subcovers of the moves and, for every history $s$ and every $m\in\omega$, put
\[
 F(s,m)=\{a_{s,0},\ldots,a_{s,m}\}.
\]

Restrict the responses to these nonempty initial segments. Again finite sequences record indices and determine the next moves of the original strategy. For $q\in\Bad(\mathbb P,p)$ let
\[
 K_q=\{f\in\omega^\omega:F(f\restr n,f(n))\subseteq\mathord\downarrow q\text{ for every }n\in\omega\}.
\]

The set $K_q$ is closed. At any history $s$, let $j(s,q)$ be the first index with $a_{s,j(s,q)}\nleq q$. Then
\[
 F(s,m)\subseteq\mathord\downarrow q\quad\Longleftrightarrow\quad m<j(s,q).
\]
Thus the tree of finite histories compatible with $K_q$ has only the finitely many successors $m<j(s,q)$ at $s$. It is finitely branching, and its body is $K_q$, so $K_q$ is compact. A winning strategy would again imply that these sets cover $\omega^\omega$. Since $\mathfrak d$ is the least size of a compact cover of $\omega^\omega$~\cite[Section~II]{Hrusak}, this would require $|\mathbb P|\ge\mathfrak d$.
\end{proof}

We now improve the upper bounds using two cardinal characteristics considered by Hru\v{s}\'ak. Following~\cite[Definition~I.1]{Hrusak}, an ideal $\mathcal I$ on $D$ is $+$-Ramsey if every nonempty tree $T\subseteq D^{<\omega}$ such that, for each $s\in T$, the set $\{x\in D:s\concat\langle x\rangle\in T\}$ belongs to $\mathcal I^+$ has a branch whose range belongs to $\mathcal I^+$. As in~\cite[Section~II]{Hrusak}, let $\ra$ be the least size of an infinite almost disjoint family $\mathcal A$ for which $\mathcal I(\mathcal A)$ is not $+$-Ramsey, and let $\aT$ be the least size of a partition of $\omega^\omega$ into nonempty compact sets.

\begin{proposition}\label{prop:upper}
We have $\epsilon_1^{\cbdl}\le\ra$ and $\epsilon_{\fin}^{\cbdl}\le\aT$.
\end{proposition}
\begin{proof}
Choose an almost disjoint family $\mathcal A$ of size $\ra$ for which $\mathcal I(\mathcal A)$ is not $+$-Ramsey. A tree witnessing this failure gives a winning strategy for Player~I in $\Gone(\mathcal I(\mathcal A)^+,\mathcal I(\mathcal A)^+)$ by offering the immediate successors of the current node. Lemma~\ref{lem:adselection} and Theorem~\ref{thm:realization} give the first inequality.

For the second inequality, we adapt the construction in~\cite[Proposition~II.4]{Hrusak}, using $R,D,S_f,D_s$ from Proposition~\ref{prop:scheepers}. Since $R^\omega$ is homeomorphic to $\omega^\omega$, choose a partition $(K_\alpha)_{\alpha<\aT}$ of $R^\omega$ into nonempty compact sets, and put
\[
 B_\alpha=\bigcup_{f\in K_\alpha}S_f,
 \qquad \mathcal A=\{B_\alpha:\alpha<\aT\}.
\]

Each $B_\alpha$ is infinite. Let $T_\alpha$ be the tree of initial segments of members of $K_\alpha$. Compactness implies that $T_\alpha$ is finitely branching. Thus, for each $s\in R^{<\omega}$, only finitely many members $r$ of $R$ satisfy $s\concat\langle r\rangle\in T_\alpha$, and their union is finite, so $B_\alpha\cap D_s$ is finite. If $\alpha,\beta<\aT$ and $\alpha\ne\beta$, then $T_\alpha\cap T_\beta$ is finite. Otherwise, K\"onig's lemma would give a common branch, which would belong to $K_\alpha\cap K_\beta$ by closedness. Consequently $B_\alpha\cap B_\beta$ is finite, since its points can occur only at the finitely many common histories. Thus $\mathcal A$ is almost disjoint and has size $\aT$.

Each $D_s$ is positive, as it meets every member of $\mathcal A$ in a finite set. Player~I uses the strategy of Proposition~\ref{prop:scheepers}, offering $D_s$ and extending the history after each nonempty response. Empty responses leave the history unchanged. If there are only finitely many nonempty responses, their union is finite. Otherwise they determine $f\in R^\omega$, and the selected union $S_f$ is contained in the member $B_\alpha$ for which $f\in K_\alpha$. In either case the outcome belongs to $\mathcal I(\mathcal A)$, so Player~I wins. The second inequality follows from Lemma~\ref{lem:adselection} and Theorem~\ref{thm:realization}.
\end{proof}

Together with Theorem~\ref{thm:lower} and the inequalities $\ra\le\aT\le\mathfrak c$ from~\cite[Proposition~II.4]{Hrusak}, Proposition~\ref{prop:upper} gives
\begin{align*}
 \covM&\le\epsilon_1^{\pos}\le\epsilon_1^{\cbdl}\le\ra\le\mathfrak c,\\
 \mathfrak d&\le\epsilon_{\fin}^{\pos}\le\epsilon_{\fin}^{\cbdl}\le\aT\le\mathfrak c.
\end{align*}

Thus all four minima lie between $\aleph_1$ and $\mathfrak c$ and are small cardinals in the sense of~\cite{Blass}.

\section{Further comments and questions}\label{sec:questions}\label{sec:stone}

The Boolean algebra cases above have topological counterparts. We record the translation through Stone duality, which also explains the relation with dense families of open sets. Let $B$ be a nontrivial Boolean algebra and let $K=\operatorname{Ult}(B)$ be its Stone space. Recall from~\cite[Theorem~7.8]{Koppelberg} that $K$ is a compact zero-dimensional Hausdorff space and that the map $b\mapsto\widehat b$, where
\[
 \widehat b=\{u\in K:b\in u\},
\]
is a Boolean isomorphism from $B$ onto $\operatorname{Clop}(K)$, whose members form a base for $K$.

For every $A\subseteq B$, we have
\begin{equation}\label{eq:stone_density}
 A\in\Vone\quad\Longleftrightarrow\quad
 \overline{\bigcup_{a\in A}\widehat a}=K.
\end{equation}

Indeed, if the union is not dense, some nonempty clopen set $\widehat b$ misses it, so $\neg b<1$ is an upper bound for $A$. Conversely, if $A\notin\Vone$, then $A$ has an upper bound $c<1$, and the nonempty clopen set $\widehat{\neg c}$ misses the union.

Following~\cite{BCPT}, write $\mathcal D_o(K)$ for the collection of nonempty families of open subsets of $K$ whose union is dense in $K$. Thus~\eqref{eq:stone_density} identifies the covers in $B$ with the clopen families in $\mathcal D_o(K)$, and the same correspondence carries selections and plays of the games to one another. To pass to arbitrary open families, for each $\mathcal U\in\mathcal D_o(K)$ put
\[
 \mathcal R(\mathcal U)=\{C\in\operatorname{Clop}(K):C\subseteq U
 \text{ for some }U\in\mathcal U\}.
\]

This family has the same union as $\mathcal U$. For each $C\in\mathcal R(\mathcal U)$, choose $r_{\mathcal U}(C)\in\mathcal U$ containing $C$. Replacing selected clopen sets by these containing members preserves density of their union.

Scheepers proved the following characterizations for arbitrary topological spaces~\cite{ScheepersV}. For Stone spaces, they follow from the Boolean algebra case of Corollaries~\ref{cor:menger_classes} and~\ref{cor:distributive}.

\begin{corollary}\label{cor:stone_games}
Let $K$ be a Stone space and $\#\in\{1,\fin\}$. Then $\mathsf S_\#(\mathcal D_o(K),\mathcal D_o(K))$ holds if and only if Player~I has no winning strategy in $\mathsf G_\#(\mathcal D_o(K),\mathcal D_o(K))$.
\end{corollary}

The formulation of $\Ufin$ in Section~\ref{sec:ufin} also has a direct interpretation here. For finite selections $(F_n)_{n\in\omega}$ in $B$, requiring $\bigcup_{n\in J}F_n\in\Vone$ for every infinite subset $J$ of $\omega$ means that every nonempty open subset of $K$ meets $\bigcup_{b\in F_n}\widehat b$ for all but finitely many $n$. Indeed, infinitely many failures give an infinite set of innings whose selected union is not dense, and the converse follows by applying~\eqref{eq:stone_density} to each infinite set of innings.
The same refinement by clopen sets preserves this eventual intersection condition, so the game characterization in Theorem~\ref{thm:ufin} has a corresponding formulation for dense open families on Stone spaces. In particular, $\Sone(\mathcal D_o(K),\mathcal D_o(K))$ is the property called \emph{selectively ccc}, see~\cite{Aurichi} and~\cite[Definition~2.3]{ASZ}. Thus Stone duality also allows one to study selective ccc through $\Sone(\Vone,\Vone)$ on Boolean algebras, and the same correspondence applies to the other dense-family selection principles considered in~\cite{BCPT}. We shall not pursue these questions here.

The results in Sections~\ref{sec:finite} and~\ref{sec:one} give sufficient conditions for the two game characterizations. So, let $\mathcal H$ be the class of infinite lattices with greatest element $1$ for which $\Sfin(\Vone,\Vone)$ implies $\Inwin\Gfin(\Vone,\Vone)$, and let $\mathcal P$ be the corresponding class for $\Sone$ and $\Gone$. Since the reverse implications hold in general, membership in either class means that the corresponding selection principle and the absence of a winning strategy for Player~I hold simultaneously or fail simultaneously.

\begin{question}\label{que:game_classes}
Characterize the classes $\mathcal H$ and $\mathcal P$.
\end{question}

Theorems~\ref{thm:menger} and~\ref{thm:rothberger} give sufficient distributive conditions for membership in these classes. By Theorem~\ref{thm:lower}, every infinite lattice of cardinality less than $\mathfrak d$ belongs to $\mathcal H$, and every infinite lattice of cardinality less than $\covM$ belongs to $\mathcal P$. For lattices of cardinality at least these respective bounds, can $\SC_1$ and $\CD_{\omega,1}$ be replaced by strictly weaker sufficient conditions for $\mathcal H$ and $\mathcal P$? Example~\ref{ex:binary_pawlikowski} shows that $\SC_1$ alone does not suffice for $\mathcal P$.

\begin{question}
Are the lower bounds $\covM$ and $\mathfrak d$ attained in ZFC for the corresponding minima? Can the minima for posets and complete bounded distributive lattices be different? Do these minima coincide with classical cardinal characteristics of the continuum?
\end{question}

The complete lattices in Section~\ref{sec:realizations} preserve the selection principles and the winning strategies of the corresponding ideals without increasing the size of the almost disjoint family. It remains to determine whether other ideal constructions admit complete lattice representations with a similar control of cardinality.

\end{document}